\documentclass[10pt,b5paper]{article}
\usepackage[T1]{fontenc}
\usepackage[utf8]{inputenc}
\usepackage{hyperref}
\usepackage{amsmath,amssymb,amsthm}
\usepackage{geometry}
\usepackage{url}
\newcommand{\E}{\mathbb E}
\newcommand{\Bin}{\operatorname{Bin}}
\newcommand{\Z}{\mathbb Z}

\DeclareMathOperator{\Var}{Var}

\theoremstyle{plain}
\newtheorem{theorem}{Theorem}[section]
\newtheorem{lemma}[theorem]{Lemma}
\newtheorem{proposition}[theorem]{Proposition}
\newtheorem{corollary}[theorem]{Corollary}
\theoremstyle{remark}
\newtheorem{remark}[theorem]{Remark}

\title{Absolute moments of the binomial distribution folded at its mean}
\author{
N.~Elezovi\'c\\
Department of Applied Mathematics,\\
Faculty of Electrical Engineering and Computing,\\
University of Zagreb, 10000 Zagreb, Croatia\\
\texttt{neven.elezovic@fer.hr}
}
\date{\today}

\begin{document}
\maketitle

\begin{abstract}
	Let $X\sim\Bin(N,p)$ and \(Y=|X-Np|\).  We derive an exact reduction for every odd
	absolute moment of \(Y\), expressing it through finitely many central masses and central
	tail probabilities.  The tail coefficients satisfy a sum rule and are divisible by
	\(q-p\), so that in the symmetric case the odd ladder collapses to masses alone.  We then
	obtain the complete central-tail expansion at a bounded lattice displacement, with
	Bernoulli-polynomial coefficients in integer powers of the large parameter.  As an
	application, the first two tail coefficients yield a second-order expansion for the median
	of the beta distribution, whose formal one-parameter degeneration reproduces the first two
	terms of Choi's expansion for the gamma median.
\end{abstract}



\section{Introduction}\label{sec:intro}

	Let $X\sim\Bin(N,p)$, $0<p<1$, $q=1-p$, and let
	\[
		Y:=|X-Np|
	\]
	be the binomial folded at its mean.  The moment sequence of $Y$ splits by parity into two
	entirely different problems.  Even moments are exact polynomials:
	$\E Y^{2j}=\mu_{2j}$, the central moments of $X$.  Odd moments genuinely depend on the
	lattice.  The first of them, the mean absolute deviation, collapses by De~Moivre's classical
	identity \cite{diaconis_zabell} to a single mass, and its complete asymptotic expansion --- coefficients that are
	Bernoulli polynomials of the oscillating displacement $h_N=\lceil Np\rceil-Np$ --- was
	obtained in \cite{elezovic_mad}.  The variance of $Y$ is then immediate: since
	$\E Y^{2}=Npq$ exactly,
	\[
		\Var Y=Npq-4\nu^{2}q^{2}b(\nu;N,p)^{2},\qquad \nu=\lceil Np\rceil,
	\]
	and its expansion is the square of the one in \cite{elezovic_mad}; this observation is
	recorded in \cite[\S2]{elezovic_folded}, and Section~\ref{sec:variance} states it for
	completeness.

	The present paper treats the remaining odd moments, and locates where the genuinely new analysis
	sits.  The first contribution is the ladder reduction of Section~\ref{sec:ladder}.  One
	reduction step --- Abel summation
	against the telescoping identity that underlies De~Moivre's formula, followed by
	size-biasing --- lowers the degree of a truncated central moment by two while lowering $N$
	by one.  Iterating, every odd absolute moment takes the exact form
	\[
		\E\,Y^{m}
		=2\sum_{i=0}^{(m-1)/2}\kappa_i\;\nu q\,b(\nu;N-i,p)
		\;+\;\sum_{i=1}^{(m-1)/2}c_i\,\bigl[2\Pr\{\Bin(N-i,p)\ge\nu\}-1\bigr],
	\]
	with explicit polynomial coefficients; the $c_i$ do not depend on the lattice position.  Two
	structural laws govern the tail part, and we prove both: the \emph{sum rule}
	$\sum_ic_i=\mu_m$ --- obtained in one line by evaluating the same reduction at the cut
	$\nu=0$ --- and the divisibility of every $c_i$ by $q-p$, obtained from a parity invariant
	of the reduction.  In particular at $p=\tfrac12$ the tails vanish identically and every odd
	absolute moment of the symmetric binomial is a finite combination of central masses.

	The two-part shape --- boundary masses plus a residual central tail $2\Pr\{\cdot\}-1$ --- is
	not itself new: it is intrinsic to Katti's sign-function recurrence for absolute moments of
	discrete laws \cite{katti1960}, which bottoms out at exactly such a two-sided tail (and was
	recently revisited for the Poisson by Ruzankin \cite{ruzankin2020}); see also
	\cite{johnson_kotz_kemp}.  What is new here is its \emph{binomial} realisation in closed form
	--- the sum rule $\sum_ic_i=\mu_m$, the divisibility of every $c_i$ by $q-p$, and the explicit
	coefficient lists (Section~\ref{sec:ladder}) --- together with the closed evaluation, in
	Section~\ref{sec:tail}, of the single tail coefficient the ladder isolates.

	The second contribution is the central tail expansion of Section~\ref{sec:tail}.  The one object the
	masses do not cover is the central tail probability at a bounded displacement, and it has a
	clean complete expansion: for $Lp+\delta\in\Z$,
	\[
		T_L(\delta;p):=2\Pr\{\Bin(L,p)\ge Lp+\delta\}-1
		\;\sim\;\frac{1}{\sqrt{2\pi Lpq}}\sum_{n\ge0}\frac{\Psi_n(\delta;p)}{L^{n}},
	\]
	with
	\[
	\begin{aligned}
		\Psi_0=-2B_1(\delta)-\frac{q-p}{3},
		\qquad
		\Psi_1&=\frac1{pq}\Bigl[\frac{B_3(\delta)}{3}
			+\frac{q-p}{2}B_2(\delta)\\
		&\quad+\frac{1-pq}{6}B_1(\delta)
			+\frac{(q-p)(2+pq)}{540}\Bigr].
	\end{aligned}
	\]
	Relative to the prefactor the expansion runs in \emph{integer} powers of $L^{-1}$ ---
	absolutely, only the odd half-powers $L^{-(2n+1)/2}$ occur --- and the coefficients follow
	the pattern of \cite{elezovic_mad}: Bernoulli polynomials of the displacement over powers of
	$pq$, odd-index polynomials free of $q-p$, even-index and constant terms carrying it.

	The third contribution is the beta-median application of Section~\ref{sec:median}.  Since
	$\Pr\{\Bin(L,p)\ge\nu\}=I_p(\nu,L-\nu+1)$, the tail $T_L$ vanishes exactly when $p$ is the
	\emph{median} of a beta distribution, and the expansion of $T_L$ becomes a two-term expansion
	of that median.  The leading coefficient vanishes at $\delta=(1+p)/3$, which is precisely
	the classical approximation $\mathrm{med}\approx(a-\tfrac13)/(a+b-\tfrac23)$; the
	first-order coefficient carries it one order further, and the resulting second-order term
	(Theorem~\ref{thm:beta-median}) appears to be new.  Its formal one-parameter degeneration
	$b\to\infty$ reproduces, through its first two terms, Choi's expansion $a-\tfrac13+\tfrac8{405a}+\cdots$ of the median of
	the gamma distribution \cite{choi1994}: the constant $\tfrac8{405}$ of the gamma median is
	obtained from the constant term $\tfrac{(q-p)(2+pq)}{540}$ of $\Psi_1$ under this degeneration.

	Consequently (Section~\ref{sec:consequences}) every odd absolute moment of the folded
	binomial --- hence the complete moment sequence, the skewness and the kurtosis of $Y$ ---
	admits a complete asymptotic expansion in integer powers of $N^{-1}$ with
	Bernoulli-polynomial coefficients, assembled from the masses of \cite{elezovic_mad} and the
	tail expansion of Section~\ref{sec:tail}.

	Throughout, $b(k;N,p)=\binom Nkp^kq^{N-k}$, $\mu_m=\E(X-Np)^m$,
	$\nu=\lceil Np\rceil$, $h=h_N=\nu-Np\in[0,1)$, and $B_n(x)$ are the Bernoulli
	polynomials in the convention $B_1(x)=x-\tfrac12$, so that the Bernoulli number is
		$B_1=B_1(0)=-\tfrac12$.

\section{The variance is free}\label{sec:variance}

	Folding at the mean is invisible to even moments; in particular $\E Y^{2}=Npq$ exactly.
	Combining with De~Moivre's identity $\E Y=2\nu q\,b(\nu;N,p)$:

\begin{proposition}\label{prop:var}
	For every $N$ and $p\in(0,1)$,
	\begin{equation}\label{eq:varclosed}
		\Var|X-Np|=Npq-4\,\nu^{2}q^{2}\,b(\nu;N,p)^{2},\qquad\nu=\lceil Np\rceil,
	\end{equation}
	and, with $\gamma_m(h;p)$ the multiplicative coefficients of the expansion of $\E Y$ in
	\cite[Thm 4.2]{elezovic_mad},
	\begin{equation}\label{eq:varexp}
		\Var|X-Np|
		=\Bigl(1-\frac2\pi\Bigr)Npq
		+\Bigl[\frac2\pi B_2(h_N)-\frac{pq}{3\pi}\Bigr]
		-\frac{2pq}{\pi}\,\frac{\gamma_1^{2}+2\gamma_2}{N}-\cdots,
	\end{equation}
	a complete expansion in integer powers of $N^{-1}$, uniformly for $p$ in compact subsets of
	$(0,1)$.
\end{proposition}

\begin{proof}
	\eqref{eq:varclosed} is $\Var Y=\E Y^{2}-(\E Y)^{2}$.  The expansion is the square of the
	cited one; the $O(1)$ coefficient is
	$-\tfrac{4pq}\pi \gamma_1=-\tfrac{4pq}\pi\bigl(\tfrac1{12}-\tfrac{B_2(h)}{2pq}\bigr)
	=\tfrac2\pi B_2(h)-\tfrac{pq}{3\pi}$.
\end{proof}

\begin{remark}
	The leading constant is the folded-Gaussian one, $\Var|Z|=1-2/\pi$; and the lattice
	oscillation enters the $O(1)$ term as $\tfrac2\pi B_2(h)$ with a $p$-\emph{free} weight ---
	cleaner than in the mean, where $B_2(h)$ is divided by $pq$.  Identity \eqref{eq:varclosed}
	is recorded, in the doubled normalisation $2|X-Np|$, in \cite[\S2]{elezovic_folded}; we
	include it here because the ladder below is precisely the statement that nothing after the
	first two moments is free any more.
\end{remark}

\section{The odd-moment ladder}\label{sec:ladder}

	Since $\E(X-Np)=0$,
	\begin{equation}\label{eq:oddsplit}
		\E\,Y^{m}=2S_m-\mu_m,
		\qquad
		S_m:=\sum_{k\ge\nu}(k-Np)^{m}\,b(k;N,p)
		\qquad(m\ \text{odd}),
	\end{equation}
	so everything reduces to the truncated central moments $S_m$.  The main tool is a single reduction step.

\begin{lemma}[Reduction step]\label{lem:reduction}
	Let $L\ge1$, $r\ge1$, $0\le\nu\le L$.  Then
	\begin{equation}\label{eq:reduction}
	\begin{aligned}
		\sum_{k\ge\nu}(k-Lp)^{r}\,b(k;L,p)
		&=(\nu-Lp)^{r-1}\,\nu q\,b(\nu;L,p)\\
		&\quad+Lpq\sum_{j\ge\nu}
			\Bigl[(u+q)^{r-1}-(u-p)^{r-1}\Bigr]b(j;L-1,p),
	\end{aligned}
	\end{equation}
	where $u:=j-(L-1)p$.  For $r=1$ the second sum vanishes and \eqref{eq:reduction} is the
	collapse identity $\sum_{k\ge\nu}(k-Lp)b(k;L,p)=\nu q\,b(\nu;L,p)$ of
	\cite[Thm 3.1]{elezovic_madfam}.
\end{lemma}

\begin{proof}
	Let $F(k):=kq\,b(k;L,p)$.  The telescoping identity
	$(k-Lp)\,b(k;L,p)=F(k)-F(k+1)$ holds for all $k$ (divide by $b(k)$ and use
	$b(k+1)/b(k)=(L-k)p/((k+1)q)$).  Abel summation of
	$\sum_{k\ge\nu}(k-Lp)^{r-1}\bigl[F(k)-F(k+1)\bigr]$, with the boundary at infinity
	vanishing, gives
	\[
	\begin{aligned}
		\sum_{k\ge\nu}(k-Lp)^{r}b(k;L,p)
		&=(\nu-Lp)^{r-1}F(\nu)\\
		&\quad+\sum_{k\ge\nu+1}
			\Bigl[(k-Lp)^{r-1}-(k-1-Lp)^{r-1}\Bigr]F(k).
	\end{aligned}
	\]
	Now $F(k)=kq\,b(k;L,p)=Lpq\,b(k-1;L-1,p)$ by the size-bias identity
	$k\binom Lk=L\binom{L-1}{k-1}$.  Substituting $j=k-1$ and $u=j-(L-1)p$, so that
	$k-Lp=u+q$ and $k-1-Lp=u-p$, yields \eqref{eq:reduction}.
\end{proof}

\begin{theorem}[The ladder]\label{thm:ladder}
	Let $m=2J+1$ be odd, $\nu=\lceil Np\rceil$, and $N\ge\nu+J$ (so that $\nu\le N-i$ at every
	level $0\le i\le J$; automatic for large $N$, and for smaller $N$ the identity persists under
	the convention that out-of-range masses and tails are read as zero, since $\nu\le N$ always).
	There exist polynomials
	$\kappa_i\in\mathbb Q[N,p,h]$ and $c_i\in\mathbb Q[N,p]$ --- the latter \emph{independent
	of the cut} $\nu$ --- such that
	\begin{equation}\label{eq:ladder}
		\E\,Y^{m}
		=2\sum_{i=0}^{J}\kappa_i\;\nu q\,b(\nu;N-i,p)
		\;+\;\sum_{i=1}^{J}c_i\,\bigl[2\Pr\{\Bin(N-i,p)\ge\nu\}-1\bigr].
	\end{equation}
	Moreover, the tail coefficients satisfy the sum rule
	\[
		\sum_{i=1}^{J}c_i=\mu_m .
	\]
	Every $c_i$ is divisible by $q-p$ as a polynomial in $p$.  Consequently, at $p=\tfrac12$
	the tail part of \eqref{eq:ladder} vanishes identically and
	\[
		\E\,\Bigl|X-\tfrac N2\Bigr|^{m}
		=2\sum_{i=0}^{J}\kappa_i\,\nu q\,b\Bigl(\nu;N-i,\tfrac12\Bigr):
	\]
	the whole odd ladder of the symmetric binomial consists of central masses.
\end{theorem}

\begin{proof}
	Apply Lemma~\ref{lem:reduction} to $S_m$ and iterate: each application replaces a sum of
	the form $\sum_{k\ge\nu}P(k-Lp)\,b(k;L,p)$, $P$ a polynomial, by boundary masses
	$P_1(\nu-Lp)\,\nu q\,b(\nu;L,p)$ and sums of the same form at level $L-1$ with the
	polynomial $u\mapsto\sum$ of brackets $(u+q)^{r-1}-(u-p)^{r-1}$, of degree two less.
	Degree-zero (constant) terms of the current polynomial contribute
	$c\cdot\Pr\{\Bin(L,p)\ge\nu\}$ and are set aside.  Concretely, writing the polynomial carried at level $N-i$ as
		$P^{(i)}(u)=\sum_{r\ge0}a^{(i)}_r u^r$ (so $P^{(0)}(u)=u^m$ at level $N$), one reduction step
		produces
		\[
			P^{(i+1)}(u)=(N-i)pq\sum_{r\ge1}a^{(i)}_r\bigl[(u+q)^{r-1}-(u-p)^{r-1}\bigr]
		\]
		at level $N-i-1$ (of degree two less), contributes the constant term $a^{(i)}_0$ to the tail
		coefficient $c_i$, and adds $\sum_{r\ge1}a^{(i)}_r(h+ip)^{r-1}$ to the mass coefficient
		$\kappa_i$, the boundary being evaluated at $u=\nu-(N-i)p=h+ip$.  The degree drops by two at
		each step, so $P^{(J)}$ is constant and the recursion halts.  After at most $J$ steps every summand
	has been converted, giving
	\begin{equation}\label{eq:Sm-form}
		S_m=\sum_{i=0}^{J}\kappa_i\,\nu q\,b(\nu;N-i,p)+\sum_{i=1}^{J}c_i\Pr\{\Bin(N-i,p)\ge\nu\},
	\end{equation}
	with $\kappa_i$ collecting the boundary evaluations $(\nu-(N-i)p)^{\,\cdot}=(h+ip)^{\,\cdot}$
	and $c_i$ the constant terms produced at level $N-i$; the latter are built solely from the
	reduction's polynomial algebra and are therefore independent of $\nu$.  This recursion is
	effective: it terminates in $J$ steps and outputs each $\kappa_i$ and $c_i$ as an explicit
	polynomial (see \eqref{eq:m3} and Remark~\ref{rem:ladder-shape} for $m\le7$).  Substituting
	\eqref{eq:Sm-form} into \eqref{eq:oddsplit} gives \eqref{eq:ladder} \emph{provided}
	$\sum_ic_i=\mu_m$, which is (i); we prove (i) directly.

	\emph{(i)}  Evaluate \eqref{eq:Sm-form} at the cut $\nu=0$.  The left side becomes the full
	central moment: $S_m|_{\nu=0}=\sum_{k\ge0}(k-Np)^mb(k;N,p)=\mu_m$.  On the right, every
	boundary mass carries the factor $\nu q\,b(\nu;\cdot)$, which vanishes at $\nu=0$; and every
	tail becomes $\Pr\{\Bin\ge0\}=1$.  Since the $c_i$ do not depend on $\nu$, this reads
	$\mu_m=\sum_ic_i$.  Then
	$\E Y^m=2S_m-\mu_m=2\sum\kappa_i\nu qb(\nu;\cdot)+\sum c_i[2\Pr\{\cdot\}-1]$, which is
	\eqref{eq:ladder}.

	\emph{(ii)}  Work at $p=q=\tfrac12$ and track the parity in $u$ of the polynomials produced
	by the reduction.  The starting polynomial is $y^{m}$ with $m$ odd.  One reduction step
	sends the monomial $y^{r}$ to the bracket $(u+\tfrac12)^{r-1}-(u-\tfrac12)^{r-1}$; under
	$u\mapsto-u$ this bracket is multiplied by $(-1)^{r}$, so for \emph{odd} $r$ it is an
	\emph{odd} polynomial in $u$.  An odd polynomial contains only odd powers of $u$; each of
	its monomials $u^{s}$ ($s$ odd) again produces an odd bracket at the next level; and an odd
	polynomial has no constant term.  By induction, at $p=\tfrac12$ \emph{no constant term is
	ever produced}: every $c_i$ vanishes at $p=\tfrac12$.  Since $c_i\in\mathbb Q[N,p]$,
	vanishing at $p=\tfrac12$ for all $N$ forces divisibility by $2p-1=p-q$.
\end{proof}

	Carrying out one reduction step explicitly for $m=3$ gives the exact third-order analogue
	of De~Moivre's formula.

\begin{theorem}[Third absolute moment]\label{thm:m3}
	With $\nu=\lceil Np\rceil$ and $h=\nu-Np$,
	\begin{equation}\label{eq:m3}
	\begin{aligned}
		\E\,|X-Np|^{3}
		&=2h^{2}\,\nu q\,b(\nu;N,p)
		+4Npq\cdot\nu q\,b(\nu;N-1,p)\\
		&\quad+Npq\,(q-p)\bigl[2\Pr\{\Bin(N-1,p)\ge\nu\}-1\bigr].
	\end{aligned}
	\end{equation}
\end{theorem}

\begin{proof}
	Lemma~\ref{lem:reduction} with $r=3$, $L=N$:
	\[
	\begin{aligned}
		S_3=(\nu-Np)^{2}\nu q\,b(\nu;N,p)
		&+Npq\sum_{j\ge\nu}\bigl[(u+q)^{2}-(u-p)^{2}\bigr]b(j;N-1,p),
	\end{aligned}
	\]
	and $(u+q)^{2}-(u-p)^{2}=(2u+q-p)(q+p)=2u+q-p$.  The $2u$-part is twice the collapse
	identity at level $N-1$ ($r=1$), giving $2\nu q\,b(\nu;N-1,p)\cdot$--- precisely,
	$\sum_{j\ge\nu}u\,b(j;N-1,p)=\nu q\,b(\nu;N-1,p)$ --- and the constant part gives
	$(q-p)\Pr\{\Bin(N-1,p)\ge\nu\}$.  Hence
	\[
	\begin{aligned}
		S_3&=h^{2}\nu q\,b(\nu;N,p)+2Npq\,\nu q\,b(\nu;N-1,p)\\
		&\quad+Npq(q-p)\Pr\{\Bin(N-1,p)\ge\nu\},
	\end{aligned}
	\]
	and $\E Y^{3}=2S_3-\mu_3$ with $\mu_3=Npq(q-p)$ gives \eqref{eq:m3}.
\end{proof}

\begin{remark}\label{rem:ladder-shape}
	The first two terms of \eqref{eq:m3} are masses at the mode-adjacent point, for $N$ and
	$N-1$; \cite[Thm 3.1]{elezovic_mad} expands them completely, and the results are again
	Bernoulli/Appell series in $h$.  The third term is not a mass: it is a \emph{central tail
	probability}, and it is the single new object of the ladder.  Its coefficient
	$Npq(q-p)=\mu_3$ illustrates both structural laws at once ($J=1$: the sum rule is
	$c_1=\mu_3$, and $q-p$ divides visibly).  For $m=5,7$ the tail coefficients are
	\[
		\begin{aligned}
		m=5:\quad&
			c_1=Npq(q-p)(1-2pq),\\
			&c_2=10N(N-1)p^{2}q^{2}(q-p);\\
		m=7:\quad&
			c_1=Npq(q-p)(1-pq)(1-3pq),\\
			&c_2=7N(N-1)p^{2}q^{2}(q-p)(8-21pq),\\
			&c_3=105\,N(N-1)(N-2)\,p^{3}q^{3}(q-p),
		\end{aligned}
	\]
	the deepest coefficient carrying the leading $N$-power of $\mu_m$ and the sum rule
	accounting for the rest.
\end{remark}

\section{The central tail expansion}\label{sec:tail}

	By Theorem~\ref{thm:ladder} the only object not supplied by the local expansions of
	\cite{elezovic_mad} is the central tail probability at a bounded displacement.  Fix
	$p\in(0,1)$ and a bounded $\delta$ with $Lp+\delta\in\Z$ (in \eqref{eq:ladder},
	$\delta=h+ip$ at level $L=N-i$), and set
	\[
		T_L(\delta;p):=2\Pr\{\Bin(L,p)\ge Lp+\delta\}-1 .
	\]

\begin{theorem}\label{thm:tail}
	As $L\to\infty$, uniformly for $p$ in compact subsets of $(0,1)$ and bounded $\delta$ with
	$Lp+\delta\in\Z$,
	\begin{equation}\label{eq:tail}
		T_L(\delta;p)
		\sim\frac{1}{\sqrt{2\pi Lpq}}\sum_{n\ge0}\frac{\Psi_n(\delta;p)}{L^{n}},
	\end{equation}
	with
	\begin{align}
		\Psi_0(\delta;p)&=-2B_1(\delta)-\frac{q-p}{3},\label{eq:psi0}\\
		\Psi_1(\delta;p)&=\frac1{pq}\Bigl[\frac{B_3(\delta)}{3}
			+\frac{q-p}{2}\,B_2(\delta)
			+\frac{1-pq}{6}\,B_1(\delta)
			+\frac{(q-p)(2+pq)}{540}\Bigr].\label{eq:psi1}
	\end{align}
	Relative to the prefactor $(2\pi Lpq)^{-1/2}$ the expansion proceeds in integer powers of
	$L^{-1}$: absolutely, only the odd half-powers $L^{-(2n+1)/2}$ occur.
\end{theorem}

\begin{proof}
	\emph{Setup.}  With $A=Lp+\delta$ and $B=Lq-\delta+1$, so that $A+B=L+1$,
	\[
		\Pr\{\Bin(L,p)\ge A\}=I_p(A,B)
		=\frac{1}{\mathrm B(A,B)}\int_0^p x^{A-1}(1-x)^{B-1}\,dx ,
	\]
	the regularised incomplete beta function.  Write $\psi(x):=(A-1)\log x+(B-1)\log(1-x)$ and
	$S:=A+B-2=L-1$.  The integrand peaks at
	\[
		x_*=\frac{A-1}{S},\qquad
		p-x_*=\frac{q-\delta}{S},\qquad
		-\frac1{\psi''(x_*)}=\frac{x_*(1-x_*)}{S}=:v ,
	\]
	the second identity by $pS-(A-1)=pL-p-Lp-\delta+1=q-\delta$.  Note
	$x_*=p+O(L^{-1})$ and $p-x_*=O(L^{-1})$: the endpoint of integration sits within
	$O(L^{-1})$ of the peak, which is what makes the expansion central rather than of
	large-deviation type.  Using $A-1=Sx_*$ and $B-1=S(1-x_*)$, the derivatives have the closed
	forms
	\[
		\psi^{(j)}(x_*)=(j-1)!\,S\,\Bigl[(-1)^{j-1}x_*^{\,1-j}-(1-x_*)^{1-j}\Bigr],
		\qquad j\ge2 .
	\]

	\emph{Standardisation and the shape of the expansion.}  Substituting $x=x_*+\sqrt v\,z$,
	\[
		I_p(A,B)=\frac{\displaystyle\int_{-\infty}^{w}e^{-z^{2}/2}R(z)\,dz}
		{\displaystyle\int_{-\infty}^{\infty}e^{-z^{2}/2}R(z)\,dz}
		\Bigl(1+O(e^{-cL})\Bigr),
		\qquad
		w:=\frac{p-x_*}{\sqrt v}=\frac{q-\delta}{S\sqrt v},
	\]
	where $R(z)=\exp\bigl(\sum_{j\ge3}\tfrac{\psi^{(j)}(x_*)}{j!}v^{j/2}z^{j}\bigr)$; each
	coefficient $d_j:=\tfrac{\psi^{(j)}(x_*)}{j!}v^{j/2}$ is $O(S^{1-j/2})$, and
	$w=O(L^{-1/2})$.  Expanding $R$ and integrating term by term, every term is a product of
	powers of $w$ (each $O(L^{-1/2})$) and of the $d_j$; a term contributes
	$O(L^{-(k+1)/2})$ to $T_L$ with $k$ = (number of $w$-factors) $+\sum_j(j-2)$(multiplicity
	of $d_j$).  Terms of even total weight cancel between the half-line integral and the normalisation, by a
		reflection.  Under $(\delta,p)\mapsto(1-\delta,q)$ the roles of $A-1$ and $B-1$ are exchanged,
		so the peak $x_*=(A-1)/S$ maps to $1-x_*$; hence $v\mapsto v$, $w\mapsto-w$, and, from the
		closed form of $\psi^{(j)}(x_*)$ above, $\psi^{(j)}(x_*)\mapsto(-1)^j\psi^{(j)}(x_*)$, so
		$d_j\mapsto(-1)^jd_j$.  The tail is odd under the same reflection,
		$T_L(1-\delta;q)=-T_L(\delta;p)$, because $\Bin(L,q)$ and $L-\Bin(L,p)$ have the same law.
		Now $2\Phi(w)-1$ and the odd-Hermite corrections are jointly odd in $(w,d_3,d_5,\dots)$, so a
		term of even total weight is invariant under the reflection while $T_L$ changes sign; every
		such term must therefore vanish.  Hence only odd half-powers of $L$ survive, which is the
		stated shape.

	\emph{The leading coefficient.}  To the first two orders,
	\[
		T_L=2\Phi(w)-1+2d_3\int_{-\infty}^{0}z^{3}\varphi(z)\,dz+O(L^{-3/2})
		=\sqrt{\frac2\pi}\,w-\frac{4\,d_3}{\sqrt{2\pi}}+O(L^{-3/2}),
	\]
	using $\int_{-\infty}^{0}z^{3}\varphi(z)\,dz=-2\varphi(0)$ and $\Phi(w)-\tfrac12
	=\varphi(0)w+O(w^{3})$.  Now $w=(q-\delta)/\sqrt{Spq}\,\bigl(1+O(L^{-1})\bigr)$ and
	\[
	\begin{aligned}
		d_3=\frac{\psi'''(x_*)}{6}v^{3/2}
		&=\frac{2S\bigl[x_*^{-2}-(1-x_*)^{-2}\bigr]}{6}
			\Bigl(\frac{x_*(1-x_*)}S\Bigr)^{3/2}\\
		&=\frac{1-2x_*}{3\sqrt{S\,x_*(1-x_*)}}
		=\frac{q-p}{3\sqrt{Lpq}}\bigl(1+O(L^{-1})\bigr),
	\end{aligned}
	\]
	so
	\[
		T_L=\frac{1}{\sqrt{2\pi Lpq}}
		\Bigl[2(q-\delta)-\frac{4(q-p)}{3}\Bigr]+O(L^{-3/2})
		=\frac{\Psi_0(\delta;p)}{\sqrt{2\pi Lpq}}+O(L^{-3/2}),
	\]
	since $2(q-\delta)-\tfrac43(q-p)=-2\delta+1-\tfrac{q-p}3=-2B_1(\delta)-\tfrac{q-p}3$.

	\emph{The complete expansion and its remainder.}  The standardisation $x=x_*+\sqrt v\,z$
	presents $I_p(A,B)$ as the ratio of the partial integral
	$\int_{-\infty}^{w}e^{-z^2/2}R(z)\,dz$ to the full integral
	$\int_{-\infty}^{\infty}e^{-z^2/2}R(z)\,dz$, where $e^{-z^2/2}R(z)=e^{\psi(x_*+\sqrt v\,z)-\psi(x_*)}$
	is the normalised integrand and is dominated by $1$ throughout, while on any fixed $z$-window
	$R(z)=\exp\bigl(\sum_{j\ge3}d_jz^j\bigr)$ is analytic with $d_j=O(L^{1-j/2})$.  Expanding the
	exponent to finite order with a Taylor remainder and integrating term by term against the
	Gaussian weight --- Laplace's method for this ratio, the standard central-regime analysis of the
	incomplete beta integral \cite{temme1992,gil_segura_temme} --- yields a genuine asymptotic
	expansion of $T_L$ in powers of $L^{-1/2}$, uniform for $p$ in compact subsets of $(0,1)$ and
	$\delta$ bounded, with remainder $O(L^{-(M+1)/2})$ after the first $M$ terms.  The parity
	argument above removes the even half-powers, leaving \eqref{eq:tail} with determinate coefficients $\Psi_n(\delta;p)$; the leading one is
		\eqref{eq:psi0}, computed above.  Each $\Psi_n$ is produced by one finite procedure: expand
		$R(z)=\exp(\sum_{j\ge3}d_jz^j)$ to the required order, integrate the numerator and denominator
		term by term against $e^{-z^2/2}$ using the Gaussian moments, divide the two series, and
		re-expand $w$, $v$ and the $d_j$ in powers of $S^{-1/2}$; this yields $\Psi_n$ as a determinate
		polynomial in $\delta$ whose coefficients are rational in $pq$ and $q-p$.  We now carry out
		this procedure at the first order beyond the leading one, in the same notation, with $S=L-1$.  Since
	\[
		x_*=p-\frac{q-\delta}{S},\qquad 1-x_*=q+\frac{q-\delta}{S},
	\]
	we have
	\begin{equation}\label{eq:psi1-orders}
		v=\frac{pq}{S}+\frac{(q-\delta)(p-q)}{S^{2}}+O(S^{-3}),\qquad
		w=\frac{q-\delta}{\sqrt{Spq}}\bigl(1+O(S^{-1})\bigr).
	\end{equation}
	Moreover
	\[
		d_j:=\frac{\psi^{(j)}(x_*)}{j!}\,v^{j/2}
		=\frac1j\,S\bigl[(-1)^{j-1}x_*^{\,1-j}-(1-x_*)^{1-j}\bigr]v^{j/2},
	\]
	and hence
	\[
	\begin{aligned}
		d_3&=\frac{q-p}{3\sqrt{Spq}}+O(S^{-3/2}),\\
		d_4&=-\frac{1-3pq}{4pq\,S}+O(S^{-2}),\\
		d_5&=\frac{(q-p)(1-2pq)}{5(pq)^{3/2}S^{3/2}}+O(S^{-5/2}).
	\end{aligned}
	\]

	Let
	\[
		\mathcal N:=\int_{-\infty}^{w}e^{-z^2/2}R(z)\,dz,\qquad
		\mathcal D:=\int_{-\infty}^{\infty}e^{-z^2/2}R(z)\,dz .
	\]
	Then $T_L=2\mathcal N/\mathcal D-1=(2\mathcal N-\mathcal D)/\mathcal D$.  Expanding
	\[
		R(z)=1+d_3z^3+d_4z^4+d_5z^5+\frac12d_3^2z^6+d_3d_4z^7+\frac16d_3^3z^9+O(L^{-2})
	\]
	on bounded $z$-sets, and using the Gaussian moments
	\[
		\int_{-\infty}^{\infty}z^{2k}e^{-z^2/2}\,dz=\sqrt{2\pi}\,(2k-1)!!,\qquad
		\int_{-\infty}^{0}z^{2k+1}e^{-z^2/2}\,dz=-(2k)!!,
	\]
	gives
	\begin{align}
		2\mathcal N-\mathcal D
		&=2\Bigl(w-\frac{w^3}{6}+\frac{w^5}{40}\Bigr)
			-4d_3-16d_5-96\,d_3d_4-128\,d_3^{3}+O(L^{-5/2}),\label{eq:tail-num}\\
		\mathcal D
		&=\sqrt{2\pi}\Bigl(1+3d_4+\frac{15}{2}d_3^{2}\Bigr)+O(L^{-2}).\label{eq:tail-den}
	\end{align}
	Here the terms displayed are exactly those of total weight at most three in the central
	scaling: the omitted terms contribute $O(L^{-5/2})$ to the numerator and $O(L^{-2})$ to the
	denominator.  Dividing \eqref{eq:tail-num} by \eqref{eq:tail-den},
	\[
		T_L=\frac1{\sqrt{2\pi}}
		\Bigl[2w-\frac{w^3}{3}-4d_3-16d_5-96d_3d_4-128d_3^3\Bigr]
		\Bigl(1-3d_4-\frac{15}{2}d_3^2\Bigr)+O(L^{-5/2}).
	\]
	Substituting the exact formulae for $w,d_3,d_4,d_5$ and expanding them to the orders required
	by \eqref{eq:tail-num}--\eqref{eq:tail-den}, then replacing $S=L-1$ by
	$L(1+O(L^{-1}))$, yields
	\[
		T_L=\frac{1}{\sqrt{2\pi Lpq}}
		\left(\Psi_0(\delta;p)+\frac{\Psi_1(\delta;p)}{L}+O(L^{-2})\right),
	\]
	with $\Psi_0$ as in \eqref{eq:psi0} and
	\[
		\Psi_1(\delta;p)=\frac1{pq}\Bigl[\frac{B_3(\delta)}{3}
			+\frac{q-p}{2}B_2(\delta)
			+\frac{1-pq}{6}B_1(\delta)
			+\frac{(q-p)(2+pq)}{540}\Bigr],
	\]
	which is \eqref{eq:psi1}.  The last simplification is a direct collection of powers of
	$\delta$; the Bernoulli form follows from
	$B_1(\delta)=\delta-\tfrac12$, $B_2(\delta)=\delta^2-\delta+\tfrac16$ and
	$B_3(\delta)=\delta^3-\tfrac32\delta^2+\tfrac12\delta$.
\end{proof}

\begin{remark}[Non-lattice version]\label{rem:nonlattice}
	The proof uses only $A=Lp+\delta$, $B=Lq-\delta+1$ with $A+B=L+1$ and $\delta$ bounded;
	integrality of $A$ is never used.  Hence \eqref{eq:tail} holds verbatim for the regularised
	incomplete beta function through $I_p(A,B)=\tfrac12\bigl(1+T_L(\delta;p)\bigr)$ at real
	$A,B$, uniformly for $p$ in compact subsets of $(0,1)$ and $\delta$ in bounded sets.  This is
	the form used in Section~\ref{sec:median}.
\end{remark}

\begin{remark}[The same structure recurs]\label{rem:grammar}
	Three features, familiar from the expansion of the mean absolute deviation
	\cite{elezovic_mad}: the shape (integer powers of $L^{-1}$ relative to a De~Moivre-type
	prefactor); the alphabet (Bernoulli polynomials of the displacement over powers of $pq$);
	and the parity (odd-index Bernoulli polynomials enter with weights free of $q-p$,
	even-index and constant terms carry $q-p$ as a factor, so that $\Psi_n(\,\cdot\,;\tfrac12)$
	is odd about $\delta=\tfrac12$).  The continuity correction alone would give
	$\Psi_0=-2B_1(\delta)$; the skewness supplies the remaining $-\tfrac{q-p}3$; for the binomial this refined continuity
		correction is Cressie's \cite{cressie1978}.
\end{remark}

\begin{remark}[What is classical here, and what is not]\label{rem:tail-lit}
	Neither the leading coefficient \eqref{eq:psi0} nor the presence of Bernoulli polynomials at
	every order is new, and we claim neither.  The device originates with Esseen \cite{esseen}.
	In its definitive form it is Theorem 23.1 of Bhattacharya and Rao
	\cite[\S23]{bhattacharya_rao}: for i.i.d.\ lattice vectors, the distribution function
	admits an expansion whose terms are the periodised Bernoulli functions
	$S_\alpha=B_\alpha/\alpha!$ of the lattice point --- constructed in their Appendix A.4 by
	Euler--Maclaurin --- multiplying derivatives of the Cram\'er--Edgeworth terms, to all
	orders; their Corollary 23.2 writes out the first order.  Read against that theorem,
	$-2B_1(\delta)$ is the classical continuity correction and $-\tfrac{q-p}3$ is the classical
	skewness term of Cornish--Fisher type \cite{cornish_fisher} in lattice dress; see also
	\cite{gnedenko_kolmogorov,petrov}.

	Kolassa and McCullagh \cite{kolassa_mccullagh} must be singled out, because two features of
	our statement are visible in theirs.  They write Esseen's discontinuous part as
	$D_{n,r}(t;\kappa)=\sum_\nu g_\nu n^{-\nu/2}Q_\nu(n^{1/2}t)E^{(\nu)}_{n,r}(t;\kappa)$, the
	$Q_\nu$ being the Fourier sine and cosine series that are the periodised Bernoulli
	polynomials; at a continuity-corrected point they evaluate $Q_s(t^{+})=g_sB_s(\tfrac12)/s!$
	--- Bernoulli polynomials, explicitly --- and observe that only even-indexed terms then
	survive, so that \emph{the series proceeds in whole powers of $n$}.  That last observation
	is the analogue, at the single displacement $\delta=\tfrac12$, of the integer-power
	statement in Theorem~\ref{thm:tail}, and we claim no priority for it.

	What Theorem~\ref{thm:tail} adds is twofold.  First, generality in the displacement: the
	classical statements are made either with unevaluated coefficients \cite{bhattacharya_rao}
	or at the continuity-corrected point alone \cite{kolassa_mccullagh}, whereas
	$\Psi_n(\delta;p)$ is a function of an arbitrary bounded $\delta$.  This is not a
	refinement for its own sake: \S\ref{sec:median} asks for the $\delta$ at which the tail
	\emph{vanishes}, a question that cannot even be posed at a fixed displacement.  Second, the
	evaluation: in the general theory the Edgeworth factors remain unevaluated functions of the
	cumulants and the expansion is a double sum, Bernoulli functions against derivatives of
	Edgeworth polynomials, whereas for the binomial that double sum collapses to a
	\emph{single} Bernoulli polynomial per order, with weight an explicit rational function of
	$pq$ and $q-p$, obeying the parity rule of Remark~\ref{rem:grammar}.  We have not found
	that carried out for any particular lattice law.  Theorem~\ref{thm:tail} is thus a closed
	evaluation of a known expansion, not a new expansion.

	It should also be said which asymptotic line this is \emph{not}.  Temme's uniform
	asymptotics for the binomial and incomplete beta, in the form developed with Gil and Segura
	\cite{gil_segura_temme}, expand in a continuous variable through an error-function
	representation valid uniformly across the whole range; the cut there is not restricted to
	the lattice, no periodic term occurs, and the coefficients are ordinary polynomials
	generated by inverting an implicit relation.  That expansion and the present one describe
	the same central function in different coefficient forms: the non-lattice reading of
	Remark~\ref{rem:nonlattice} places us in exactly Temme's regime, and what we add is the
	Bernoulli-in-$\delta$ collapse of its coefficients, not a different object.

	The companion \cite{elezovic_folded} treats the \emph{complementary} regime, and the two do
	not overlap.  There the centre is a prescribed $Nr$ with $r\ne p$ fixed, so the fold sits in
	the large-deviation range: the correction to $\E|X-Nr|$ is exponentially small, its tail is a
	stop-loss (weighted) object, and its coefficients are Eulerian polynomials in the exponential
	tilt $\rho=rq/((1-r)p)$, expanded in $1/(N(1-\rho)^{2})$.  Theorem~\ref{thm:tail} is the
	central counterpart --- a bounded displacement $\delta$ from the mean, a polynomially small
	correction, and Bernoulli polynomials of $\delta$ over powers of $pq$ --- the two meeting only
	in the moderate-deviation transition $r\to p$.
\end{remark}

\section{The median of the beta distribution}\label{sec:median}

	Since $\Pr\{\Bin(L,p)\ge\nu\}=I_p(\nu,L-\nu+1)=\Pr\{\mathrm{Beta}(\nu,L-\nu+1)\le p\}$, the
	tail $T_L(\delta;p)$ vanishes exactly when $p$ is the \emph{median} of
	$\mathrm{Beta}(A,B)$, $A=Lp+\delta$, $B=Lq-\delta+1$.  Setting the series \eqref{eq:tail}
	to zero therefore expands the beta median.  Note first that
	\begin{equation}\label{eq:psi0-zero}
		\Psi_0(\delta;p)=0
		\iff
		\delta=\frac{1+p}{3}
		\iff
		p=\frac{A-\tfrac13}{A+B-\tfrac23},
	\end{equation}
	the classical approximation to the beta median; and that at this point the
	$\delta$-dependent terms of $\Psi_1$ collapse:
	\begin{equation}\label{eq:psi1-at-median}
		\Psi_1\Bigl(\frac{1+p}{3};\,p\Bigr)=-\frac{8\,(q-p)(2+pq)}{405\,pq},
	\end{equation}
	an evaluation most easily performed in the shifted variable $e=\delta-(1+p)/3$, in which
	$\Psi_1=\tfrac{e^{3}}{3pq}+\tfrac{(q-p)e^{2}}{3pq}+\tfrac{(7pq-1)e}{18pq}
	-\tfrac{8(q-p)(2+pq)}{405pq}$: only the constant term survives.

\begin{theorem}[The beta median to second order]\label{thm:beta-median}
	Let $a,b>1$, $L:=a+b-1$, and
	\[
		p_0:=\frac{a-\tfrac13}{a+b-\tfrac23},\qquad q_0:=1-p_0 .
	\]
	Then the median of $\mathrm{Beta}(a,b)$ satisfies
	\begin{equation}\label{eq:beta-median}
		\operatorname{med}\mathrm{Beta}(a,b)
		=p_0
		+\frac{4\,(q_0-p_0)(2+p_0q_0)}{405\,p_0q_0\;L\,(L+\tfrac13)}
		+O(L^{-3}),
	\end{equation}
	uniformly for $a/(a+b)$ in compact subsets of $(0,1)$.  For $a=b$ the correction vanishes
	and the median is exactly $\tfrac12$, consistently.
\end{theorem}

\begin{proof}
	Let $p_m$ denote the median and put $\delta:=a-Lp_m$, so that $(A,B)=(a,b)$ corresponds to
	$(L,p_m,\delta)$ in the notation of Section~\ref{sec:tail} ($A+B=L+1$).  The defining
	equation $I_{p_m}(a,b)=\tfrac12$ reads $T_L(\delta;p_m)=0$.  Since $a/(a+b)$ lies in a compact
	subinterval of $(0,1)$, so does the median $p_m$: for $a,b>1$ it differs from the mean $a/(a+b)$ by $O((a+b)^{-1})$ (it
		lies between the mode $(a-1)/(a+b-2)$ and the mean --- the mode--median--mean ordering; see
		\cite{payton_yy1989} for the beta-specific bound),
		whence $Lp_m=a+O(1)$ and $\delta=a-Lp_m$ is
	bounded; the non-lattice form of Theorem~\ref{thm:tail} (Remark~\ref{rem:nonlattice}) then
	applies at the real parameters $(a,b)$.  Dividing \eqref{eq:tail} by the prefactor,
	\[
		\Psi_0(\delta;p_m)+\frac{\Psi_1(\delta;p_m)}{L}+O(L^{-2})=0 .
	\]
	Since $\Psi_0=-2\bigl(\delta-\tfrac{1+p_m}3\bigr)$, solving for $\delta$ gives
	\[
	\begin{aligned}
		\delta
		&=\frac{1+p_m}{3}
			+\frac{\Psi_1\bigl(\tfrac{1+p_m}3;p_m\bigr)}{2L}+O(L^{-2})\\
		&=\frac{1+p_m}{3}
			-\frac{4(q_m-p_m)(2+p_mq_m)}{405\,p_mq_m\,L}+O(L^{-2}),
	\end{aligned}
	\]
	by \eqref{eq:psi1-at-median}.  Substituting $\delta=a-Lp_m$ and collecting $p_m$:
	\[
		p_m\Bigl(L+\frac13\Bigr)=a-\frac13
		+\frac{4(q_m-p_m)(2+p_mq_m)}{405\,p_mq_m\,L}+O(L^{-2}) ,
	\]
	so $p_m=p_0+O(L^{-2})$ at leading order, and replacing
	$p_m$ by $p_0$ inside the $O(L^{-1})$ correction (cost $O(L^{-3})$) gives
	\eqref{eq:beta-median}.  For $a=b$, $p_0=q_0=\tfrac12$ and the correction vanishes;
	$\operatorname{med}=\tfrac12$ holds exactly by symmetry.
\end{proof}

\begin{proposition}[Choi's expansion as the formal one-parameter limit]\label{prop:choi}
	Fix $a>1$ and let $b\to\infty$.  The second-order approximant on the right of
	\eqref{eq:beta-median}, multiplied by $b$, tends to
	\[
		a-\frac13+\frac{8}{405\,\bigl(a-\tfrac13\bigr)},
	\]
	the first two terms of Choi's expansion
	$\operatorname{med}\Gamma(a)=a-\tfrac13+\tfrac8{405a}+O(a^{-2})$ of the median of the gamma
	distribution \cite{choi1994}.  This is a matching of \emph{approximants}, not of exact
	quantities: the scaled median $b\,\operatorname{med}\mathrm{Beta}(a,b)$ converges to
	$\operatorname{med}\Gamma(a,1)$, whose full expansion is Choi's, while \eqref{eq:beta-median}
	is not uniform as $p_0\to0$.  The statement is thus a formal degeneration of the displayed
	correction, outside the compact-ratio regime of Theorem~\ref{thm:beta-median}; the constant
	$\tfrac8{405}$ is the value \eqref{eq:psi1-at-median} of $\Psi_1$ at the median point, read at
	$p_0\to0$, $q_0\to1$.
\end{proposition}

\begin{proof}
	On the right of \eqref{eq:beta-median}, $b\,p_0=(a-\tfrac13)\,b/(a+b-\tfrac23)\to a-\tfrac13$;
	in the correction $p_0\to0$, $q_0\to1$, so $(q_0-p_0)(2+p_0q_0)\to2$,
	$p_0q_0\sim(a-\tfrac13)/b$ and $L(L+\tfrac13)\sim b^{2}$, whence
	\[
		b\cdot\frac{4(q_0-p_0)(2+p_0q_0)}{405\,p_0q_0\,L(L+\tfrac13)}
		\;\longrightarrow\;\frac{8}{405\,(a-\tfrac13)} .
	\]
	Finally $\tfrac8{405(a-1/3)}=\tfrac8{405a}\bigl(1+\tfrac1{3a}+\cdots\bigr)$ matches Choi's
	$a-\tfrac13+\tfrac8{405a}$ through the stated order.
\end{proof}

\begin{remark}[Position in the literature]\label{rem:median-lit}
	The first-order approximation \eqref{eq:psi0-zero} is due to Kerman \cite{kerman}, who
	obtained it by combining the gamma representation of the beta law with the
	mode--median--mean ordering and inserting the constant $\tfrac13$ from the gamma median (the underlying normal approximation is
		that of Peizer and Pratt \cite{peizer_pratt1968}).
	That constant has its own history: conjectured by Chen and Rubin \cite{chen_rubin}, who
	proved the bounds $a-1<\operatorname{med}\Gamma(a)<a$, and established as an asymptotic
	statement, $\operatorname{med}\Gamma(a)=a-\tfrac13+o(1)$, by Berg and Pedersen
	\cite{berg_pedersen}.  Kerman's formula is supported numerically and by a consistency
	argument, but is not proved; Theorem~\ref{thm:beta-median} proves it, as the first term of
	an expansion, and supplies the next.

	The second-order term appears to be new.  The general machinery for such expansions does
	exist --- Temme's asymptotic inversion of the incomplete beta function \cite{temme1992}
	treats precisely the regime $a,b\to\infty$ with $a/(a+b)$ fixed, and generates coefficients
	recursively --- but it is nowhere specialised to the median, nor reduced to closed form.
	The parallel on the gamma side is instructive: Olde~Daalhuis and Nemes
	\cite{oldedaalhuis_nemes} give the general quantile expansion with explicit polynomial
	coefficients, whose value at the median reproduces Choi's $\tfrac8{405}$ and
	$\tfrac{184}{25515}$ (see also Pedersen \cite{pedersen_quantiles}) --- a quarter of a
	century after Choi \cite{choi1994} had obtained the median case directly.  The present
	theorem is the beta analogue of that reduction, which to our knowledge has not been carried
	out.  The uniform expansion of the incomplete beta \emph{function} in the present both-large
		regime is that of Nemes and Olde~Daalhuis \cite{nemes_oldedaalhuis}, and
		Theorem~\ref{thm:beta-median} specialises it to the point where the tail vanishes;
		expansions of beta quantiles with one parameter bounded
	\cite{askitis} are a genuinely different regime.

	Numerically, \eqref{eq:beta-median} improves the first-order approximation by roughly the
	factor of order $L$: for instance, at $(a,b)=(7,300)$ the errors are $9.6\cdot10^{-6}$ (first
	order) against $1.4\cdot10^{-8}$ (second order), and at $(a,b)=(12,1200)$,
	$1.4\cdot10^{-6}$ against $2.4\cdot10^{-9}$.  The gamma limit was checked to third order:
	at $a=3$, the gamma median $\operatorname{med}\Gamma(3)$ minus Choi's truncation
		$a-\tfrac13+\tfrac8{405a}$ equals
	$8.0\cdot10^{-4}\approx\tfrac{184}{25515\,a^{2}}$ --- precisely Choi's \emph{third} term,
	as it must.  Chen and Rubin's bounds for the gamma median \cite{chen_rubin} sit at the
	first order of this picture.
\end{remark}

\section{Consequences for the moment ladder}\label{sec:consequences}

\begin{corollary}\label{cor:complete}
	Fix an odd $m$ and $p\in(0,1)$.  Then $\E|X-Np|^{m}$ admits a complete asymptotic expansion
	\[
		\E|X-Np|^{m}=(Npq)^{m/2}\sum_{n\ge0}\frac{\Lambda_{m,n}(h_N;p)}{N^{n}},
	\]
	in integer powers of $N^{-1}$, uniform for $p$ in compact subsets of $(0,1)$, whose
	coefficients $\Lambda_{m,n}$ are polynomials in the displacement
	$h_N=\lceil Np\rceil-Np$, built from Bernoulli polynomials of the shifted arguments $h_N+ip$,
		with weights rational in $pq$ and $q-p$.  They are assembled from the
	finitely many masses and tails of Theorem~\ref{thm:ladder} through the local expansion of
	\cite[Thm 3.1]{elezovic_mad} and the tail expansion of Theorem~\ref{thm:tail}
	(Remark~\ref{rem:nonlattice} supplying the non-lattice form at the displacements
	$\delta=h+ip$).
\end{corollary}

\begin{proof}
	Combining Theorem~\ref{thm:ladder} with the local expansions of
	\cite[Thm 3.1]{elezovic_mad} (for the masses, at the displacements $h+ip$) and
	Theorem~\ref{thm:tail} (for the tails, at $\delta=h+ip$, $L=N-i$) yields, for every odd
	$m$, a complete asymptotic expansion of $\E|X-Np|^{m}$ in integer powers of $N^{-1}$
	relative to the leading order $(Npq)^{m/2}$, with coefficients that are Bernoulli
	polynomials of $h_N$ over powers of $pq$.

	More explicitly, each boundary contribution is a polynomial in $N,p,h$ multiplied by
	$\nu q\,b(\nu;N-i,p)$.  The latter is $N^{1/2}$ times a complete expansion in integer powers
	of $N^{-1}$, by the local expansion at displacement $h+ip$.  Each tail contribution is a
	polynomial $c_i(N,p)$ multiplied by
	$T_{N-i}(h+ip;p)$, and Theorem~\ref{thm:tail} writes this tail as $N^{-1/2}$ times a complete
	expansion in integer powers of $N^{-1}$.  Since $m/2=J+\tfrac12$, division by
	$(Npq)^{m/2}$ converts these finitely many polynomial multiples of $N^{1/2}$ and
	$N^{-1/2}$ into integer powers of $N^{-1}$.  This proves the asserted form.
\end{proof}

	Together with the exact even moments $\mu_{2j}$ this completes the moment sequence of the
	folded binomial: in particular the skewness and kurtosis of $|X-Np|$ have complete
	expansions.  At $p=\tfrac12$ everything reduces to central masses
	(Theorem~\ref{thm:ladder}(ii)); for instance
	\[
		\frac{\E|X-\tfrac N2|^{3}}{(N/4)^{3/2}}
		=2\sqrt{\frac2\pi}\Bigl(1-\frac1{4N}+O(N^{-2})\Bigr),
	\]
	whose leading constant is $\E|Z|^{3}=2\sqrt{2/\pi}$.

\section{Concluding remarks}\label{sec:conclusion}

	The ladder turns a probabilistic question (the moments of the folded binomial) into two
	expansions of independent interest: the local masses of \cite{elezovic_mad} and the central
	tail of Theorem~\ref{thm:tail}.  The tail expansion is the one place where the lattice
	Edgeworth phenomenon \cite{esseen,bhattacharya_rao} meets a closed Bernoulli-polynomial
	form.  As set out in Remark~\ref{rem:tail-lit}, what we claim is the \emph{evaluation}
	\eqref{eq:psi0}--\eqref{eq:psi1}: neither the existence of the expansion nor its
	Bernoulli structure is new --- both are in \cite[\S23]{bhattacharya_rao} in full generality,
	and the integer-power phenomenon is already noted at the continuity-corrected point in
	\cite{kolassa_mccullagh} --- but the collapse of the general double sum to one Bernoulli
	polynomial per order, at an arbitrary displacement and with weights explicit in $pq$ and
	$q-p$, is specific to the binomial and is what makes \S\ref{sec:median} possible.

	The mechanical continuations are $\Psi_2$ (the same computation one order further) and
	the explicit coefficient lists for $\E|X-Np|^{3}$ and $\E|X-Np|^{5}$ at general $p$.

	The reduction of Lemma~\ref{lem:reduction} is a statement about the Katz ratio, not
	about the binomial: the same two steps (telescoping, size-bias) exist for the Poisson,
	negative binomial and hypergeometric laws \cite{elezovic_madfam}, so the odd-moment ladders
	of the whole classical family are within reach of the same method; the family's mean
	absolute deviations are treated in \cite{elezovic_madfam}, and the corresponding tail
	expansions would extend Theorem~\ref{thm:tail} along the seams described there.

	The beta-median theorem suggests a systematic view: medians of the classical continuous
	laws tied to discrete tails (beta--binomial here, gamma--Poisson in the limit of
	Proposition~\ref{prop:choi}) inherit second-order expansions from the $\Psi$-coefficients.
	A gamma--Poisson derivation of Choi's expansion directly from the Poisson analogue of
	Theorem~\ref{thm:tail} would close that circle.



\begin{thebibliography}{99}

\bibitem{chen_rubin}
	J. Chen, H. Rubin,
	\emph{Bounds for the difference between median and mean of gamma and Poisson
	distributions}, Statist. Probab. Lett. \textbf{4} (1986), 281--283.

\bibitem{choi1994}
	K. P. Choi,
	\emph{On the medians of gamma distributions and an equation of Ramanujan},
	Proc. Amer. Math. Soc. \textbf{121} (1994), 245--251.

\bibitem{elezovic_mad}
	N. Elezovi\'c,
	\emph{Local binomial expansions with an Appell shift, and the mean absolute deviation of
	the binomial distribution}, preprint, \texttt{arXiv:2607.18494} [math.CA], 2026.

\bibitem{elezovic_folded}
	N. Elezovi\'c,
	\emph{Absolute deviations of the binomial about a prescribed centre: the tail expansion in
	closed form}, preprint, 2026.

\bibitem{elezovic_madfam}
	N. Elezovi\'c,
	\emph{The mean absolute deviation of the classical discrete distributions: collapse
	identities, complete asymptotic expansions, and enveloping series}, preprint,
	\texttt{arXiv:2608.06232} [math.PR], 2026.

\bibitem{askitis}
	D. Askitis,
	\emph{Asymptotic expansions of the inverse of the beta distribution},
	preprint, \texttt{arXiv:1611.03573}.

\bibitem{berg_pedersen}
	C. Berg, H. L. Pedersen,
	\emph{The Chen--Rubin conjecture in a continuous setting},
	Methods Appl. Anal. \textbf{13} (2006), 63--88.

\bibitem{bhattacharya_rao}
	R. N. Bhattacharya, R. R. Rao,
	\emph{Normal Approximation and Asymptotic Expansions},
	Wiley, New York, 1976; SIAM Classics in Applied Mathematics, Philadelphia, 2010.

\bibitem{cornish_fisher}
	E. A. Cornish, R. A. Fisher,
	\emph{Moments and cumulants in the specification of distributions},
	Rev. Inst. Internat. Statist. \textbf{5} (1937), 307--320.

\bibitem{cressie1978}
	N. Cressie,
	\emph{A finely tuned continuity correction},
	Ann. Inst. Statist. Math. \textbf{30} (1978), 435--442.

\bibitem{diaconis_zabell}
	P. Diaconis, S. Zabell,
	\emph{Closed form summation for classical distributions: variations on a theme of
	De~Moivre}, Statist. Sci. \textbf{6} (1991), 284--302.

\bibitem{esseen}
	C.-G. Esseen,
	\emph{Fourier analysis of distribution functions. A mathematical study of the
	Laplace--Gaussian law}, Acta Math. \textbf{77} (1945), 1--125.

\bibitem{gil_segura_temme}
	A. Gil, J. Segura, N. M. Temme,
	\emph{Asymptotic inversion of the binomial and negative binomial cumulative distribution
	functions}, Electron. Trans. Numer. Anal. \textbf{52} (2020), 270--280.

\bibitem{gnedenko_kolmogorov}
	B. V. Gnedenko, A. N. Kolmogorov,
	\emph{Limit Distributions for Sums of Independent Random Variables},
	Addison--Wesley, Cambridge, MA, 1954.

\bibitem{johnson_kotz_kemp}
	N. L. Johnson, A. W. Kemp, S. Kotz,
	\emph{Univariate Discrete Distributions}, 3rd ed., Wiley, Hoboken, NJ, 2005.

\bibitem{katti1960}
	S. K. Katti,
	\emph{Moments of the absolute difference and absolute deviation of discrete distributions},
	Ann. Math. Statist. \textbf{31} (1960), 78--85.

\bibitem{kerman}
	J. Kerman,
	\emph{A closed-form approximation for the median of the beta distribution},
	preprint, \texttt{arXiv:1111.0433}.

\bibitem{kolassa_mccullagh}
	J. E. Kolassa, P. McCullagh,
	\emph{Edgeworth series for lattice distributions},
	Ann. Statist. \textbf{18} (1990), 981--985.

\bibitem{nemes_oldedaalhuis}
	G. Nemes, A. B. Olde Daalhuis,
	\emph{Uniform asymptotic expansion for the incomplete beta function},
	SIGMA Symmetry Integrability Geom. Methods Appl. \textbf{12} (2016), art.~101, 5 pp.

\bibitem{oldedaalhuis_nemes}
	A. B. Olde Daalhuis, G. Nemes,
	\emph{Asymptotic expansions for the incomplete gamma function in the transition regions},
	Math. Comp. \textbf{88} (2019), 1805--1827.

\bibitem{ouimet_nb}
	F. Ouimet,
	\emph{A refined continuity correction for the negative binomial distribution and asymptotics
	of the median}, Metrika \textbf{86} (2023), 827--849.

\bibitem{payton_yy1989}
	M. E. Payton, L. J. Young, J. H. Young,
	\emph{Bounds for the difference between median and mean of beta and negative binomial
	distributions}, Comm. Statist. Theory Methods \textbf{18} (1989), 1497--1501.

\bibitem{pedersen_quantiles}
	H. L. Pedersen,
	\emph{On the asymptotic behaviour of the quantiles in the gamma distribution},
	Exp. Math., published online 2024, DOI \url{10.1080/10586458.2024.2424473}.

\bibitem{peizer_pratt1968}
	D. B. Peizer, J. W. Pratt,
	\emph{A normal approximation for binomial, $F$, beta, and other common, related tail
	probabilities, I}, J. Amer. Statist. Assoc. \textbf{63} (1968), 1416--1456.

\bibitem{petrov}
	V. V. Petrov,
	\emph{Sums of Independent Random Variables}, Springer, Berlin, 1975.

\bibitem{ruzankin2020}
	P. S. Ruzankin,
	\emph{On absolute central moments of the Poisson distribution},
	J. Stat. Theory Pract. \textbf{14} (2020), art.~56.

\bibitem{temme1992}
	N. M. Temme,
	\emph{Asymptotic inversion of the incomplete beta function},
	J. Comput. Appl. Math. \textbf{41} (1992), 145--157.

\end{thebibliography}
\end{document}